\documentclass[12pt]{amsart}
\usepackage{amsmath}
\usepackage{latexsym}
\usepackage{amssymb}
\usepackage{graphicx}
\newtheorem{theorem}{Theorem}
\newtheorem{proposition}{Proposition}
\newtheorem{lemma}{Lemma}

\newtheorem{example}{Example}

\newtheorem{remark}{Remark}

\begin{document}
%
\def\R {{\mathbb{R}}}
\def\N {{\mathbb{N}}}
\def\C {{\mathbb{C}}}
\def\Z {{\mathbb{Z}}}
\def\phi{\varphi}
\def\epsilon{\varepsilon}
\def\ma{{\mathcal A}}
%
\def\tb#1{\|\kern -1.2pt | #1 \|\kern -1.2pt |} 
\def\Qed{\qed\par\medskip\noindent}
%
\title[Generation of singularities from the initial datum]{Generation of singularities from the initial datum for Hamilton-Jacobi equations} 
\author{Paolo Albano} 
\address{Dipartimento di Matematica, 
Universit\`a di Bologna, Piazza
di Porta San Donato 5, 40127 Bologna, Italy} 
\email{paolo.albano@unibo.it}
\author{Piermarco Cannarsa} 
\address{Dipartimento di Matematica, 
Universit\`a di Roma "Tor Vergata", Via della Ricerca Scientifica, 00133 Roma, Italy}
\email{cannarsa@mat.uniroma2.it}
\author{Carlo Sinestrari}
\address{Dipartimento di Ingegneria Civile e Ingegneria Informatica, 
Universit\`a di Roma "Tor Vergata", Via del Politecnico, 00133 Roma, Italy}
\email{sinestra@mat.uniroma2.it}

\date{\today}

\begin{abstract}
 We study the generation of singularities from the initial datum for a solution of the Cauchy problem for a class of Hamilton-Jacobi equations of evolution. For such equations, we give conditions for the existence of singular generalized characteristics starting at the initial time from a given point of the domain, depending on the properties of the proximal subdifferential of the initial datum in a neighbourhood of that point. 
\end{abstract}


\maketitle

\section{Introduction and statement of the results}
\setcounter{equation}{0}
\setcounter{theorem}{0}
\setcounter{proposition}{0}  
\setcounter{lemma}{0}
\setcounter{corollary}{0} 
\setcounter{definition}{0}

Let $\Omega \subset \R^n$ be an open subset, let $T$ be a positive number and let $u:[0,T[\times \Omega \longrightarrow \R$ be a continuous viscosity solution of the Cauchy problem
    \begin{equation}
      \label{eq:cp}
      \begin{cases}
        \partial_t u(t,x)+H(t,x, D_xu(t,x))=0,\quad \text{ in }]0,T[\times \Omega,
    \\
    u(0,x)=u_0(x), \quad \text{ in }\Omega.
        \end{cases} 
      \end{equation}
 We assume the following properties on $H$ and $u_0$. \medskip \\
 {\bf (H)} The Hamiltonian function $H=H(t,x,p)$ is of class $C^2( [0,T]\times \Omega \times \R^n)$ and the hessian w.r.t. the $p$ variable $D_{pp} H(t,x,p)$ is positive definite for all $t,x,p$.  \medskip \\
 {\bf (U$_0$)} The initial value $u_0$ is locally Lipschitz continuous on $\Omega$. \medskip 

We are interested in the local properties of $u$ for small times and away from the boundary of $\Omega$, therefore we do not make any requirement on the behaviour of $u$ on $\partial \Omega$. 


It is well known that solutions of first order Hamilton-Jacobi equations develop singularities: even if the datum $u_0$ is regular, smooth solutions in general exist only for small times. We therefore consider generalized solutions in the viscosity sense, see \cite{CIL,CL}. When the  Hamiltonian is convex in $p$, viscosity solutions are characterized by the property of being semiconcave. In particular, they are Lipschitz continuous and satisfy the equation in the classical sense at points of differentiability. The points $(t,x) \in [0,T[ \times \Omega$ where $u$ is not differentiable are called the {\em singularities} of $u$ and their union is called the singular set.

An interesting property of the singularities of solutions of Hamilton-Jacobi equations is that, under suitable hypotheses, they propagate along generalized characteristics, i.e. Lipschitz curves which solve the equation of characteristics in a generalized sense, see Section 2 for the precise definition. For a solution of \eqref{eq:cp}, we have the following statement, first proved in \cite{AC}: if $(t_0,x_0) \in ]0,T[ \times \Omega$ is singular for $u$, then there exists a generalized characteristic $\gamma:[t_0,t_1] \to \Omega$, such that $\gamma(t_0)=x_0$, and such that $(t,\gamma(t))$  belongs to the singular set of $u$ for $t$ in a right neighbourhood of $t_0$. Further local propagation results, both for evolutionary and for stationary equations, have been obtained in \cite{S1,CY,Y}. In addition, for certain classes of Hamilton-Jacobi equations, there are global propagation results, see \cite{A2,AC2,CC,CMS}, ensuring that generalized characteristics starting from a singular point remain inside the singular set for all subsequent times. Global propagation of singularities has interesting topological applications to both homotopy equivalence and contractibility of the singular set of solutions, see in particular \cite{ACNS,CCF}. 

In this paper we wish to extend this analysis to the case where $t_0=0$ and we consider the following question: \medskip

    {\bf (Q):} for a given $x_0 \in \Omega$, which properties of $u_0$ near $x_0$ ensure that there is (resp. there is not) a singular characteristic starting from $x_0$? 
    
  \medskip

We remark that, if the initial value is $C^2$, then the classical method of characteristics shows that the solution is smooth for a short time, and the above problem becomes trivial. We therefore address question {\bf (Q)} for more general initial data assuming only Lipschitz continuity. We observe that the case considered here has some new features compared with the propagation results mentioned above. In fact, the points $(t_0,x_0)$ with $t_0=0$ are peculiar not only because they lie on the boundary of the domain of $u$, but also because the time $t_0=0$ is the one where we do not have the semiconcavity of the solution, a property which is essential for proving the propagation results mentioned above.

It is interesting to observe that the nondifferentiability of $u_0$ at $x_0$ is neither a necessary nor a sufficient condition for the generation of singularities at $x_0$. To find examples, it suffices to consider the Cauchy problem for the eikonal equation in one dimension
   \begin{equation}\label{eq:1+1}
      \begin{cases}
\partial_t u(t,x)+\frac 12 |\partial_xu(t,x)|^2=0, \quad \text{ in }]0,T[\times \R \medskip
        \\
     u(0,x)=u_0(x), \quad \text{ for }x\in \R .    
        \end{cases}
      \end{equation}
Under mild assumptions on the initial value, the viscosity solution of \eqref{eq:1+1} is given by the Hopf formula
\begin{equation}\label{hopf}
u(t,x)=\inf_{y\in \R} \left[ u_0(y)+ \frac{(x-y)^2}{2t} \right].
\end{equation}

    \begin{example}\label{e1}{\rm 
      Consider equation \eqref{eq:1+1} with initial datum $u_0(x)=|x|$. Then the Hopf formula \eqref{hopf} yields    
       $$
      u(t,x)=
      \begin{cases}
        -x-\frac t2,\quad &\text{ for }x\le -t,
        \\
        \frac{x^2}{2t},\quad &\text{ for } -t<x<t ,
        \\
        x-\frac t2,\quad &\text{ for }x\geq t .
        \end{cases}
      $$
Although the initial datum is singular at $x=0$, it is easy to see that the solution is differentiable everywhere on $]0,T]\times \R$.       
     } \end{example}

      \begin{example}\label{e2}{\rm
      Consider now equation \eqref{eq:1+1}  with $u_0(x)=-|x|^\alpha$ for some fixed $1<\alpha<2$. Then $u_0 \in C^1(\R)$ but $u_0''(x) \to -\infty$ as $x \to 0$. If we consider points with $x=0$ we have, by the Hopf formula \eqref{hopf},
$$
u(0,t)= \inf_{y\in \R} \left[ -|y|^\alpha+ \frac{y^2}{2t} \right].
$$
The expression inside the brackets attains the minimum at two points $y=\pm(t\alpha)^\frac{1}{2-\alpha}$. Since the minimizer is not unique, well known properties of equations of the form \eqref{eq:cp} imply that $u$ is not differentiable at $(0,t)$ for any $t>0$. It can also be checked that $\gamma(t) \equiv 0$ is a singular generalized characteristic originating at $0$. 
}\end{example}

In this paper we show that, under assumptions (H) and (U$_0$), the regularity of a solution $u$ of \eqref{eq:cp} near a given point $(0,x_0)$ is related to the properties of $D_{pr}^- u_0(x_0)$, the proximal subdifferential of $u_0$ at $x_0$.
  We recall that $D_{pr}^- u_0(x_0)$ consists of the vectors $p$ such that, for some $K>0$, we have
   \begin{equation}
   \label{prox-intro}
   u_0(x) \geq u_0(x_0) + \langle p,x-x_0 \rangle - \frac K2 |x-x_0|^2 
   \end{equation}
for $x$ near $x_0$. The structure of this set is related to the semiconvexity property:  it can be shown that $u_0$ is semiconvex if and only if $D_{pr}^- u_0(x_0) \neq \emptyset$ and \eqref{prox-intro} is satisfied with a uniform $K$ for every $x_0$.
Examples 1 and 2 suggest that the regularity of the solution depends on the property of $D_{pr}^- u_0$ being nonempty, since $D_{pr}^- u_0(0)=[-1,1]$ in the former case and $D_{pr}^- u_0(0)=\emptyset$ in the latter. Our results confirm this intuition, although a complete answer to question (Q) has to take into account the properties of $D_{pr}^- u_0$ in a whole neighbourhood of $x_0$.

The starting point of our analysis is a one-to-one correspondence between the classical characteristics emanating from $(0,x_0)$ and the elements of $D_{pr}^- u_0(x_0)$, which is given in Theorem \ref{one-to-one}. Using, this, in Section 4 we give a criterion for the regularity of $u$ in a forward parabolic neighbourhood of a given point $(0,x_0)$ in terms of the local semiconvexity of $u_0$ near $x_0$, see Theorem \ref{t:semic}. If the local semiconvexity fails, we also show the existence of a weakly singular generalized characteristics emanating from $(0,x_0)$, that is, a characteristic contained in the closure $\overline {\Sigma (u)}$ of the singular set of $u$. We then give two examples where $u_0$ has nonempty proximal subdifferential at all points near a given $x_0$, but \eqref{prox-intro} is not satisfied with a uniform $K$. In such cases, the solution cannot be smooth in a whole forward parabolic neighbourhood of $(0,x_0)$, but the generation of singularities can occur in different ways. In the first case, Example 3, there is a unique characteristic emanating from $x_0$, which is classical but weakly singular: in fact, its points are regular but belong to the closure of the singular set. In Example 4, there are infinitely many characteristics starting from $x_0$, all of which are classical except for a singular one.
 
In Section 5, we focus our attention on a more specific class of Hamiltonians, which are quadratic in $p$ and to which the global propagation results of \cite{A2} apply. In this case, we give a result on the propagation of (strong) singularities by characterizing the generation points of singular generalized characteristics as those with empty proximal subdifferential. Finally, we show that, if the initial datum is semiconcave, our statements imply that all singularities of the initial data propagate forward in time, extending the results in the previous literature on the propagation from singular points at positive times.

\section{Preliminaries and assumptions}

Consider a function $v:\Omega \to \R$, with $\Omega \subset \R^N$ an open set, and let $x_0 \in \Omega$. 
The (Fr\'echet) {\em subdifferential} of $v$ at $x_0$ is defined as follows
$$
D^-v(x_0)=\left\{ p \in \R^n ~:~ \liminf_{x \to x_0} \frac{v(v)-v(x_0)- \langle p, x-x_0 \rangle }{|x-x_0|} \geq 0   \right\}.
$$
Given $K>0$, we say that $p_0$ is a {\em proximal $K$-subgradient} of $v$ at $x_0$ if
$$
v(x) \geq v(x_0) + \langle p_0, x-x_0 \rangle - \frac{K}{2}|x-x_0|^2,
$$
for $x$ in a neighbourhood of $x_0$. This is equivalent to saying that the pair $(p_0,K I)$, with $I$ the identity matrix, belongs to the second order subjet as defined in \cite{CIL}. The set of all proximal subgradients, that is,
$$
\begin{array}{rcl} D_{pr}^- v(x_0)  =  \{ p_0 \in \R^n & : & \mbox{$p_0$ is a proximal $K$-subgradient}  \\
&& \mbox{ of $v$ at $x_0$ for some $K>0$} \ \}
\end{array}
$$ 
is called the {\em proximal subdifferential} of $v$ at $x_0$. Clearly, $D_{pr}^- v(x_0) \subset D^-v(x_0)$ but in general the inclusion can be strict, see e.g. the initial datum $u_0$ of Example 2 at $x_0=0$.  We define in an analogous way the Fr\'echet and the proximal superdifferential of $v$, and we denote them by $D^+v$ and $D^+_{pr}v$ respectively.

If $v$ is a $C^{1,1}$ function, then its classical gradient is also a proximal gradient on both sides. More precisely, if $L$ is a Lipschitz constant for $Dv$ in a neighbourhood of a given point $x_0$, it is easy to see that $Dv(x_0)$ is a both an $L$-subgradient and an $L$-supergradient for $v$ at $x_0$.

A useful remark, which follows directly from the definition, is the following: if two functions $v,w$ are such that $v-w$ attains a local minimum at $x_0$, then any proximal $K$-subgradient of $w$ at $x_0$ is also a $K$-subgradient of $v$.

If $u=u(t,x)$ is defined in a subset of $\R \times \R^n$, we denote by $D^-_x u(t_0,x_0)$ the subdifferential of $u$ with respect to the $x$ variables, that is, the subdifferential of the function $x \mapsto u(t_0,x)$ at $x=x_0$. We define $D^+_xu(t_0,x_0)$ analogously.

We now recall the definition and main properties of semiconcave functions, and we refer to \cite{CS} for a detailed treatment.
A function $v$ is called {\em semiconcave} if it can be locally represented as $v=v_1+v_2$, with $v_1$ a concave function and $v_2 \in C^2$. It follows from the definition that a semiconcave function has non-empty proximal superdifferential at every point. As already mentioned in the introduction, a point $x$ in the domain of a semiconcave function $v$ is called {\em singular} if $v$ is non-differentiable at $x$. The set of singular points is denoted by $\Sigma(v)$ and it is a set of measure zero. Similarly, a function $v$ is called {\em semiconvex} if $-v$ is semiconcave. 

It is well known, see e.g. \cite{CL,CS}, that a viscosity solution $u$ of \eqref{eq:cp} is locally semiconcave on $]0,T[\times \Omega$, even if $u_0$ is merely Lipschitz continuous. In addition, if an open subset $V$ of $]0,T[\times \Omega$ does not contain singular points for $u$, then $u$ is $C^{1,1}_{loc}$ on $V$. For this reason, the set $\overline{\Sigma(u)}$, i.e. the closure of $\Sigma (u)$  in $]0,T[\times \Omega $, is called the $C^{1,1}$ {\em singular support} of $u$.

We recall that a $C^1$ arc $\gamma:[t_0,t_1] \to \Omega$ is called a {\em classical characteristic} associated with a solution $u$ of equation \eqref{eq:cp} if, for all $t \in ]t_0,t_1[$, the function $u$ is differentiable at $(t,\gamma(t))$ and we have
\begin{equation}\label{classicalchar}
\dot \gamma(t)=D_pH(t,\xi (t), D_xu(t,\gamma(t))).
\end{equation}
As in \cite{AC,CS}, we call a {\em generalized characteristic} associated with $u$ a curve $(t,\gamma(t)) \in [0,T[ \times \Omega$, where 
$\gamma:[t_0,t_1] \to \Omega$ is a Lipschitz function which satisfies
\begin{equation}\label{ggc}
\dot \gamma(t) \in \text{co }D_pH(t,\gamma (t),D_x^+u(t,\gamma (t)) ),\qquad \text{ for a.e. }t\in [t_0,t_1],
	\end{equation} 
where ``co'' stands for the convex hull. For simplicity, in the following we will often refer to the space component $\gamma(\cdot)$ as to the characteristic curve. We recall that there are other possible interesting generalized definitions of characteristic, which will  not be treated in this paper, e.g. the broken characteristics considered in \cite{KS,S1,SA}.

We have the following result on the existence of generalized characteristics, see \cite{AC,CY,Y}.
\begin{theorem}\label{gen-char}
Let $u$ be a semiconcave solution of \eqref{eq:cp}. Then, for any $(t_0,x_0) \in  ]0,T[\times \Omega$, there exists at least one generalized characteristic $\gamma:[t_0,\sigma[ \to \Omega$, where either  ${\rm dist}(\gamma(t),\partial \Omega) \to 0$ as $t \to \sigma$ or $\sigma =T$.
\end{theorem}

\begin{remark}\label{initial}{\rm
Even if no assumption is made on the semiconcavity of the initial data, it is easy to see that the conclusion of Theorem \ref{gen-char} also holds in the case $t_0=0$. In fact, we can obtain a generalized characteristic starting from $(0,x_0)$ as the limit of characteristics $\{ \gamma_h\}$ originating from points $(t_h,x_0)$, with $t_h \downarrow 0$, using a standard compactness argument, see e.g. \cite{Y}.
}\end{remark}

Since the results of our paper are local, they do not depend on the behaviour at infinity of $H$ and $u_0$. For this reason, it is convenient to perform part of our analysis under some additional assumptions on the data which will be removed later. In the rest of the section, we restrict ourselves to the case where $\Omega=\R^n$ and we assume that $H$ is a Tonelli Hamiltonian in the sense of \cite{CY}, namely it satisfies:  \medskip

\noindent {\bf (H1)} (Uniform convexity) There exists a nonincreasing function $\nu:[0,+\infty[ \to ]0,+\infty[$ such that $H_{pp}(t,x,p) \geq \nu(|p|) I$, for all $t,x,p$, where $I$ is the identity matrix. \\
\noindent {\bf (H2)} (Superlinear growth) There exist two superlinear functions $\theta, \bar \theta: [0,+\infty[ \to [0,+\infty[$ and a constant $c_0$ such that $\theta(|p|) -c_0 \leq H(t,x,p) \leq \bar \theta (|p|)$, for all $t,x,p$. \\
\noindent {\bf (H3)} (Uniform regularity) There exists a nondecreasing function $K:[0,+\infty[ \to [0,+\infty[$ such that $H$ and all its first and second derivatives are bounded by $K(|p|)$, for all $(t,x,p)$. \medskip

        In addition, we assume

        \noindent {\bf (U$_0$*)} the initial datum $u_0:\R^n \to \R$ is globally Lipschitz. \medskip

We underline that these assumptions will be removed later and that the main results of the paper, except for the ones in Section 5, only require the conditions (H) and (U$_0$) stated in the introduction.

        Under hypotheses (H1)---(H3) and (U$_0$*), it is well known that problem \eqref{eq:cp} has a unique viscosity solution $u$, that can be represented as
\begin{equation}\label{valuef}  u(t,x)   = 
 \inf_{\stackrel{\xi \in W^{1,1}([0,t])}{\xi(t)=x}}\left( u_0( \xi(0) )+\int_0^t L(\tau,\xi(\tau),\dot \xi (\tau) )\, d\tau \right),  \end{equation}
for all $(t,x)\in ]0,T]\times \R^n$. Here the Lagrangian function $L$ is given by
$$
L(t,x,q)=\max_{p \in \R^n} \left[ p \cdot q - H(t,x,p) \right].
$$
If $\xi$ is any minimizer in \eqref{valuef}, the function $p(s):=D_q L(s, \xi(s),\dot \xi(s))$ is called the {\em dual arc} associated with $\xi(\cdot)$. We recall the following properties,  see e.g. \cite[Theorems 6.3.3 and 6.4.8]{CS}. 

\begin{proposition}\label{classical}
Let $x \in \R^n$ and $t>0$ be fixed. Then the infimum in \eqref{valuef} is attained. If $\xi(\cdot)$ is a minimizer, and if $p(\cdot)$ is the associated dual arc, then $\xi,p$ are both of class $C^1$ and solve the hamiltonian system
\begin{equation}\label{ham-syst}
\left\{ \begin{array}{l}
\dot \xi(s)= D_p H(s,\xi(s),p(s)) \medskip\\
\dot p(s)= -D_x H(s,\xi(s),p(s)).
\end{array}\right.
\end{equation}
The function $u$ is differentiable at $(t,x)$ if and only if the minimizer in \eqref{valuef} is unique.
In addition, $u$ is differentiable at $(s,\xi(s))$ for all $s \in ]0,t[$ and satisfies $D_xu(s,\xi(s))=p(s)$.
\end{proposition}

 It turns out that the minimizing arcs in \eqref{valuef} are also classical characteristics, as recalled in the next statement.

\begin{proposition}\label{reg}
Let $\xi:[0,t_0] \to \R^n$ be a Lipschitz arc. Then $\xi$ is a minimizer in \eqref{valuef}, with $(t,x)=(t_0,\xi(t_0))$, if and only if $\xi$ is a classical characteristic associated with $u$.
\end{proposition}
\begin{proof}
The property that a minimizer is a classical characteristic follows from the previous proposition. The converse implication is obtained by a direct computation. Suppose that $\xi$ is a classical characteristic, and set $p(t)=D_x u(t,\xi(t))$, for $t \in ]0,t_0[$. Then 
$\dot \xi(t)=H_p(t,\xi(t),p(t))$ and therefore, by well-known properties of the Legendre transform,
$$
L(t,\xi(t),\dot \xi(t))+ H(t,\xi(t),p(t)) = \dot \xi(t) \cdot p(t).
$$
Since $u$ satisfies the equation in the classical sense at the points of differentiability, it follows
\begin{eqnarray*}
\frac{d}{dt} u(t,\xi(t)) & = & u_t(t,\xi(t)) + p(t) \cdot \dot \xi(t) \\
& = & - H(t,\xi(t),p(t)) + p(t) \cdot \dot \xi(t)= L(t,\xi(t),\dot \xi(t)).
\end{eqnarray*}
We conclude
$$
u(t_0,\xi(t_0))=u_0(\xi(0)) + \int_0^{t_0} L(t,\xi(t),\dot \xi(t)) dt,
$$
which implies that $\xi$ is a minimizer in \eqref{valuef}.
\end{proof}

We now recall the definition and the basic properties of the action functional associated with our problem, see e.g. \cite{B,CC}. We define, for given $x,y \in \R^n$ and $t>0$,
\begin{equation}\label{fund-sol}
\ma_t(y,x)=\inf_{\stackrel{\xi \in W^{1,1}([0,t])}{\xi(0)=y, \xi(t)=x}} \left\{ \int_0^t L(s,\xi(s),\dot \xi(s)) \, ds 
\right\}.
\end{equation}

Under our hypotheses, the infimum in \eqref{fund-sol} is a minimum. As before, any minimizer $\xi(\cdot)$ is of class $C^1$ and satisfies, together with its dual arc $p(s):=D_q L(s, \xi(s),\dot \xi(s))$, the hamiltonian system \eqref{ham-syst}. 
We can restate \eqref{valuef} as
\begin{equation}\label{valuef2}
 \qquad  u(t,x)   =  \inf_{y \in \R^n} \left( u_0( y )+ \ma_t(y,x) \right). 
 \end{equation}
Clearly, an arc $\xi(\cdot)$ is a minimizer in \eqref{valuef} if and only if, after setting $y=\xi(0)$, $\xi(\cdot)$ is a minimizer in \eqref{fund-sol} and $y$ is a minimizer in \eqref{valuef2}.

We need the following properties of $\ma$, see Proposition 2.2 in \cite{B}, Propositions B.8 and B.9 in \cite{CC}. The analysis in \cite{CC} is performed in the case of a Hamiltonian which does not depend on $t$; however, the proofs extend to our setting in a straightforward way.

\begin{lemma}\label{action}
For any $\Lambda_0>0$, there exist $t_0 >0$ and $c_0>0$ such that, for any fixed $t \in ]0,t_0]$ and $x \in \R^n$, the following properties hold. \\
{\bf (i)}
For any $y$ in the ball $B_{ \Lambda_0 t}(x)=\{ y \ : \ |x-y| <  \Lambda_0 t \}$, the hamiltonian system \eqref{ham-syst} with endpoint conditions for $\xi$
\begin{equation}\label{bound}
\xi(0)=x, \qquad \xi(t)=y
\end{equation}
has a unique solution $(\xi(\cdot),p(\cdot))$. In addition, the infimum in the definition \eqref{fund-sol} of $\ma_t(y,x)$ is attained by a unique arc $\xi(\cdot)$, which is the space component of the solution to \eqref{ham-syst}--\eqref{bound}.
 \\
{\bf (ii)} 
The function $y \to \ma_t(y,x)$  is of class $C^{1,1}$ in the ball $ B_{ \Lambda_0 t}(x)$, with $C^{1,1}$ norm only depending on $\Lambda_0,t$. The derivative is given by
$$
D_y {\mathcal A}_t(y,x) = -D_q L(0, \xi(0), \dot \xi(0)),
$$
where $\xi(\cdot)$ is the unique minimizer in the definition of ${\mathcal A}_t(y,x)$. \medskip \\
{\bf (iii)}
The function
$
y \to \ma_t(y,x) - \dfrac{c_0}{t}|y|^2
$
is convex in the ball $ B_{ \Lambda_0 t}(x)$.
\end{lemma}

\section{A one-to-one correspondence}

The following result shows that there is a one-to-one correspondence between the elements of $D_{pr}^- u_0(y_0)$ and the classical characteristics starting at $(0,y_0)$. 

\begin{theorem}\label{one-to-one}
Under hypotheses {\rm (H)}, {\rm (U$_0$)}, let $u \in {\rm Lip}_{loc}([0,T] \times \Omega)$ be a solution of problem \eqref{eq:cp}. Then, for any $y_0 \in \Omega$, the following properties hold.\\
{\bf (i)} Given $p_0 \in D_{pr}^- u_0(y_0)$, let $(\xi(\cdot),p(\cdot))$ be the solution of \eqref{ham-syst} with initial conditions $\xi(0)=y_0, \quad p(0)=p_0$. Then there exists $\tau_0>0$ such that $\xi(\cdot)$ is a classical characteristic on $[0,\tau_0]$. \\
{\bf (ii)} Conversely, let $\xi:[0,\tau_0] \to \Omega$ be a classical characteristic starting at $y_0$. Then there exists a unique $p_0  \in D_{pr}^- u_0(y_0)$ such that $\xi(\cdot)$ is the first component the solution of \eqref{ham-syst} with initial conditions $\xi(0)=y_0, \quad p(0)=p_0$.
\end{theorem}

\begin{proof}
Let us first prove the result in the case where $\Omega=\R^n$, and the additional assumptions (H1)---(H3) and (U$_0^*$) hold. In the final part of the proof, we will see how the general case can be reduced to this one by a localization argument.

It is well known that, under the global Lipschitz assumption on $u_0$, the solution $u$ is also globally Lipschitz. We then set
\begin{equation}\label{lambdazero}
\Lambda_0 = 2 \sup \{ D_p H(t, x, p) \ : \ (t,x) \in [0,T[ \times \R^n, |p| \leq {\rm Lip}(u) \}.
\end{equation}

Let $p_0$ be a proximal subgradient of $u_0$ at $y_0$. By definition, there exist $K,r>0$ such that
\begin{equation}\label{subgr}
u_0(y) \geq u_0(y_0) + \langle p_0, y-y_0 \rangle - \frac K2 |y-y_0|^2, \qquad \mbox{ for all $y \in B_r(y_0)$}.
\end{equation}

Now let $(\xi(\cdot),p(\cdot))$ be the solution of \eqref{ham-syst} with initial conditions $\xi(0)=y_0, \quad p(0)=p_0$. Since
$|p_0| \leq {\rm Lip}(u)$, we have
$$|\dot \xi(0)| = |D_p H(0, x_0, p_0)| \leq \Lambda_0/2,
$$ and so there exists $\tau_0 >0$ such that 
\begin{equation}\label{speed}
|\dot \xi(t)| \leq \Lambda_0, \qquad  \forall t \in [0,\tau_0].
\end{equation}
Let $t_0, c_0$ be the constants associated with $\Lambda_0$ in Lemma \ref{action}.
We can assume that $\tau_0$ above is chosen small enough to satisfy
\begin{equation}\label{tzero}
\tau_0 \leq t_0, \qquad \Lambda_0 \tau_0 \leq r, \qquad \frac{c_0}{\tau_0} > \frac K2,
\end{equation}
where $r,K$ are the constants in \eqref{subgr}.

We now set $x_0=\xi(\tau_0)$.  By our choice of $\tau_0$ and by Lemma \ref{action}(i), we know that the minimizer for $\ma_{\tau_0}(y_0,x_0)$ is unique and coincides with the arc $\xi$ defined above. In addition, by \eqref{speed} and by Lemma \ref{action}(ii), the function $y \to \ma_{\tau_0}(y,x_0)$ is differentiable at $y=y_0$ and satisfies
$$
\left.  D_y \ma_{\tau_0} (y,x_0) \right\rvert_{y=y_0} = - D_q L(0, \xi(0),\dot \xi(0)) = - p(0) = - p_0.
$$
The convexity property of Lemma \ref{action}(iii) gives
\begin{eqnarray*}
{\mathcal A}_{\tau_0}(y,x_0) & = & {\mathcal A}_{\tau_0}(y,x_0) - \frac{c_0}{\tau_0}|y|^2 + \frac{c_0}{\tau_0}|y|^2 \\
& \geq & {\mathcal A}_{\tau_0}(y_0,x_0) - \frac{c_0}{\tau_0}|y_0|^2 + \langle D_y {\mathcal A}_{\tau_0}(y_0,x_0), y-y_0 \rangle \\
&&-2\frac{c_0}{\tau_0} \langle y_0, y-y_0 \rangle + \frac{c_0}{\tau_0}|y|^2 \\
& = & {\mathcal A}_{\tau_0}(y_0,x_0) - \langle p_0, y-y_0 \rangle + \frac{c_0}{\tau_0}|y-y_0|^2 
\end{eqnarray*}
for any $y$ such that $|y-x_0| \leq \Lambda_0 \tau_0$. Together with inequalities \eqref{subgr}  and \eqref{tzero}, this implies
\begin{eqnarray}
\lefteqn{u_0(y)+{\mathcal A}_{\tau_0}(y,x_0) -u_0(y_0) - {\mathcal A}_{\tau_0}(y_0,x_0)} \nonumber \\
& \geq & \langle p_0,y-y_0 \rangle - \frac K2 |y-y_0|^2 -  \langle p_0,y-y_0 \rangle + \frac{c_0}{\tau_0}|y-y_0|^2 \nonumber \\
\qquad & = & \left( \frac{c_0}{\tau_0} - \frac K2 \right) |y-y_0|^2>0, \label{aconv}
\end{eqnarray}
for all $y \neq y_0$ such that $|y-x_0| \leq \Lambda_0 \tau_0$. 

We now claim that $\xi(\cdot)$ is the unique minimizer in \eqref{valuef} with endpoint $(x,t)=(x_0,\tau_0)$. In fact, let $\zeta:[0,\tau_0] \to \R^n$ be any minimizer for this problem.  By Proposition \ref{reg}, $\zeta$ is a classical characteristic. In particular, by our choice of $\Lambda_0$, it satisfies $|\dot \zeta(t) | \leq \Lambda_0/2$. Then, if we set $z_0=\zeta(0)$ we find that $|x_0-z_0|=|\zeta(\tau_0)-\zeta(0)| \leq \frac 12 \Lambda_0 \tau_0$. On the other hand, the minimality of $\zeta$ implies that $z_0$ is a minimizer for \eqref{valuef2} with $(x,t)=(x_0,\tau_0)$. In view of \eqref{aconv}, we have $z_0=y_0$. Since $\zeta$ and $\xi$ are both solutions of \eqref{ham-syst} with the same endpoints, they coincide by Lemma \ref{action}(i). This shows that $\xi(\cdot)$ is the unique minimizer 
 and therefore a classical characteristic on $[0,\tau_0]$, by Proposition \ref{reg}.

Conversely, let $\xi:[0,t_0] \to \R^n$ be a classical characteristic starting at $y_0$. By Proposition \ref{reg}, $\xi$ is a minimizer in \eqref{valuef} for the point $(t,x)=(t_0,\xi(t_0))$, which implies
$$
u(t_0,\xi(t_0))=u_0(y_0)+\ma_{t_0}(y_0,\xi(t_0)) = \min_{y \in \R^n} [u_0(y) + \ma_{t_0}(y,\xi(t_0))].
$$
From this we see that the function  $y \to u_0(y) + \ma_{t_0}(y,\xi(t_0))$ attains a  minimum at the point $y=y_0$. By our choice of $\Lambda_0$ in \eqref{lambdazero}, we know that $|\dot \xi(t)| \leq \Lambda_0/2$ for all $t$, and therefore $|y_0-\xi(t_0)| \leq \frac 12 \Lambda_0 t_0$.  By the $C^{1,1}$ regularity of $\ma$ given by Lemma \ref{action}(ii), we deduce that the proximal subgradient of $u_0$ at $y_0$ is nonempty and contains the vector
 $$
 p_0 := - D_y \ma_{t_0}(y,\xi(t_0)) = D_q L(0,\xi(0),\dot \xi(0))=p(0),
 $$
 where $p$ is the dual arc associated with $\xi$. Therefore, the pair $(\xi,p)$ solves system \eqref{ham-syst} with initial conditions $y(0)=y_0, p(0)=p_0$, as it was claimed in (ii). We also observe that the value of $p_0$ in (ii) is uniquely determined by the condition $\dot\xi(0)=D_p H(0,y_0,p_0)$ and by the strict convexity of $H$. This proves the theorem in the more restrictive setting described at the beginning of the proof.

Let us now drop the additional assumptions and consider the case of a general open set $\Omega \subset \R^n$ and general $H,u_0$ satisfying only (H) and (U$_0$). Let us fix any $y_0 \in \Omega$, and let $R>0$ be such that $\overline{B_{2R}(y_0)} \subset \Omega$. Denote by $L_0$ the Lipschitz constant of $u$ in $[0,T] \times \overline{B_{2R}(y_0)}$ and set
$$
L_1=\max \{ H(t,x,p) ~:~ (t,x,p) \in [0,T] \times \overline{B_{2R}(y_0)} \times \overline{B_{L_0}(0)} \, \}.
$$
Now, let $\tilde H \in C^2([0,T] \times \R^n \times \R^n)$ be a function satisfying (H1)---(H3) and such that $\tilde H \equiv H$ on the set $[0,T] \times {B_{2R}(y_0)} \times{B_{L_0}(0)}$ and let $\tilde u_0:\R^n \to \R$ be a function which coincides with $u_0$ on $B_{2R}(y_0)$ and is globally Lipschitz with constant $L_0$. We then denote by $\tilde u$ the solution of  \eqref{eq:cp} on $[0,T] \times \R^n$, with $H,u_0$ replaced by $\tilde H, \tilde u_0$. By standard uniqueness results for viscosity solutions in cones of propagation, see e.g. \cite[Theorem V.3]{CL}, we see that $u \equiv \tilde u$ on $[0,T_0] \times B_R(y_0)$ provided $T_0>0$ satisfies $L_1 T_0 \leq R$. Since $\tilde u$ satisfies the hypotheses of the first part of the proof, the assertion follows.
\end{proof}

Let us consider the solution $u$ of Example 1. In that case, we have $D_{pr}^-u_0=[-1,1]$. For any $p_0 \in [-1,1]$, the solution of \eqref{ham-syst} with initial conditions $\xi(0)=0, p(0)=p_0$ is $\xi(t)=p_0 t$. From the explicit form of the solution, we can see directly that $\xi(t)$ is a classical characteristic for all $t>0$. This infinite family of characteristics emanating from $y=0$, which induces the instantaneous regularization of the solution, is called a ``rarefaction wave'' in the context of hyperbolic conservation laws.

\section{A criterion for local regularity}

Let us consider a solution of \eqref{eq:cp} under the general assumptions (H) and (U$_0$). We take a point $y_0 \in \Omega$ and we wish to analyze how the regularity of $u$ near $(0,y_0)$ is influenced by the behaviour of $u_0$ in a neighbourhood of $y_0$.

From Theorem \ref{one-to-one} we know that, if the proximal subdifferential of $u_0$ at $y_0$ is nonempty, then there exists at least a characteristic curve starting at $y_0$ along which the solution is smooth. This suggests that, if the same property holds in a neighbourhood of $y_0$, the solution $u$ should be smooth in a whole forward parabolic neighbourhood of $(0,y_0)$. The next result shows that this is indeed the case, provided subdifferentiability holds in a uniform way, and that the condition is necessary and sufficient.

\begin{theorem}\label{t:semic}
Under assumptions {\rm (H)}, {\rm (U$_0$)}, let $u:[0,T] \times \Omega \to \R$ be a viscosity solution of \eqref{eq:cp}, and let $y_0 \in \Omega$. Then the following properties are equivalent. \\
{\bf (i)}  There exists $K>0$ such that $u_0$ has a proximal $K$-subgradient for all $y$ in a neighbourhood of $y_0$. \\
{\bf (ii)} There exist $R,t_0>0$ such that  $u \in C^{1,1}_{loc} (]0,t _0] \times B_{R}(y_0))$.
      \end{theorem} 
\begin{proof} 
We first prove that (i) implies (ii). It is sufficient to consider the case where $\Omega=\R^n$, assumptions (H1)---(H3) and (U$_0^*$) hold, 
 and $u_0$ has a $K$-subdifferential everywhere, since the general case can be reduced to this one as in the last step of the proof of Theorem \ref{one-to-one}.

It is easy to see that the $K$-subdifferentiability at every point of $u_0$ implies the semiconvexity of $u_0$ with constant $K$ (in fact, the two properties are equivalent). Then, it is known, see Proposition 4.10 in \cite{B}, that $u(t,\cdot)$ is also semiconvex for $t>0$ small enough. On the other hand, our solution is locally semiconcave on $]0,T] \times \Omega$. This shows that $u(t,\cdot)$ is both semiconcave and semiconvex, hence $C^{1,1}$, for small $t$. Solutions of \eqref{eq:cp} which are regular in space are also jointly regular in space and time, see the argument of Corollary 7.3.5 in \cite{CS}.

Let us now prove the converse implication. Again, it is not restrictive to assume properties (H1)---(H3) and (U$_0^*$).
 Let $L_0$ be the Lipschitz constant of $u$ in $[0,t_0] \times B_R(y_0)$,  set  
$$
\Lambda_0 = \sup \{ D_p H (t, x, p)  :  (t,x,p) \in [0,t_0] \times B_R(y_0)\times B_{L_0}(0) \},
$$
and  set also
$$ \tau_0=\min\left\{ \frac{R}{2 \Lambda_0} , t_0\right\},$$
where $t_0$ is the constant associated to $\Lambda_0$ in Lemma~\ref{action}. Let us now fix any $y\in B_{R/2}(y_0)$. By Remark \ref{initial}, there exists a generalized characteristic $\gamma_y (\cdot)$ such that  $\gamma_y (0)=y$. By our choice of $\Lambda_0, \tau_0$, we have that $|y-\gamma (t)|\le \Lambda_{0} t \leq R/2$ for all $t \in [0, \tau_0]$. Because of assumption (ii), the points $(t,\gamma_y(t))$ are regular for $t \in ]0,\tau_0]$ and so $\gamma_y(\cdot)$ is a classical characteristic. By Proposition \ref{reg}, we have
%
          $$   u(\tau_0,\gamma_y(\tau_0))=u_0(y)+\ma_{\tau_0}(y,\gamma_y (\tau_0))=
\min_{z\in \R^n} [u_0(z)+\ma_{\tau_0}(z,\gamma_y (\tau_0))].
         $$
By the $C^{1,1}$ regularity of $\ma$ given in Lemma~\ref{action}(ii), we deduce that $-D_y\ma_{\tau_0}(y,\gamma_y(\tau_0))$ is a $K$-proximal subgradient for $u_0$ at $y$, with $K$ only depending on $\Lambda_0,\tau_0$. Since these constants do not depend on $y \in B_{R/2}(y_0)$, the conclusion follows.
\end{proof}

Observe that an initial datum $u_0$ may well satisfy property (i) of the above statement without being differentiable at $y_0$, see Example 1. In such cases, the solution exhibits instantaneous smoothing for small times.

We can give a counterpart to the previous theorem by showing that, if condition (i) does not hold, then we have weak generation of singularities from $y_0$ in the following sense: there exists a generalized characteristic $\gamma (\cdot )$ such that $(t,\gamma (t))\in \overline{\Sigma (u)}$ for every $t$ in the domain of definition of $\gamma$. In particular, $u$ may be differentiable along $\gamma$ (that is, $\gamma$ may be a classical characteristic) but it is not $C^1$ in any neighbourhood of $\gamma$.

\begin{theorem} \label{weak}
 Under assumptions {\rm (H)}, {\rm (U$_0$)}, let $u:[0,T] \times \Omega \to \R$ be a solution of \eqref{eq:cp}, and let $y_0 \in \Omega$. Then, the following assertions are equivalent. 
\\
{\bf (i)} For every $V$ neighbourhood of $y_0$, and for every $K>0$ there exists $y\in V$ such that $u_0$ has no proximal $K$-subgradient at $y$. \\
{\bf (ii)} There exists a generalized characteristic $\gamma$, defined on $[0,\sigma [$ (with either $\sigma =T$ or $\gamma (\sigma )\in \partial \Omega$), such that $\gamma (0)=y_0$ and $(t,\gamma (t))\in \overline{\Sigma (u)}$, for every $t\in [0,\sigma [$. 
\end{theorem} 
\begin{proof}
 Assume (i). By Theorem~\ref{t:semic}, for every $t_0>0$ and $R>0$ the solution $u$ of \eqref{eq:cp} is not $C^{1,1}$ in $]0,t_0[\times B_R(y_0)$, and therefore the intersection $( \, ]0,t_0[\times B_R(y_0))\cap \Sigma (u)$ is nonempty.

This implies that we can find a sequence $(t_h,y_h)\in\Sigma (u)$, with $t_h>0$ and $y_h\in \Omega$, converging to $(0,y_0)$.           
By Theorem 1.1 in \cite{A1}, for every $y_h$ there exists a generalized characteristic $\gamma_h$, with $\gamma_h (t_h)=y_h$, such that
$(t,\gamma_h (t))\in \overline{\Sigma (u)}$ for every $t\in [t_h, \sigma_h [$, with either $\sigma_h=T$ or $\gamma_h(\sigma_h)\in \partial\Omega$. By compactness of the generalized characteristics, we can take a limit $\gamma= \lim_{h\to\infty}\gamma_h$ which satisfies (ii) above.

If we assume instead that (i) does not hold, then Theorem~\ref{t:semic} implies that $u$ is $C^{1,1}_{loc}(\, ]0,t_0[\times B_R(y_0))$ for some $t_0,R>0$, and this shows that (ii) cannot hold.    
\end{proof} 

In the light of our results, it is interesting to consider cases where the initial data has nonempty proximal subdifferential at every point, but the subdifferentiability does not hold with a uniform $K$ around a certain point $y_0$. In fact, on the one hand, Theorem \ref{one-to-one} gives the existence of a classical characteristic starting at $y_0$. On the other hand, Theorem \ref{weak} implies the generation of a weakly singular characteristic from the same point. We give two examples where this mixed behaviour can be recovered explicitly.

\begin{example} \label{e3}{\rm
Consider equation \eqref{eq:1+1} with $u_0(y)=y^2 \sin \frac 1y$ and $u_0(0)=0$. For $y \neq 0$, such a function is smooth and satisfies
$$
u_0'(y)= 2y \sin \frac 1y - \cos \frac 1y, \quad u_0''(y)= \left( 2 +\frac{1}{y^2} \right)\sin \frac 1y - \frac 2y \cos \frac 1y.
$$
The function is differentiable also at $y=0$ with $u_0'(0)=0$, but $u_0'$ is not continuous at $y=0$. From the definition, we see that $p_0=0$ is a proximal $K$-subgradient at $y=0$, with $K=2$. By $C^2$-regularity for $y \neq 0$, we deduce that $D_{pr}^- u_0(y)$ is nonempty for every $y \in \R$. However, since $\liminf_{y \to 0} u_0''(u) = -\infty$, we see that proximal subdifferentiability does not hold with a uniform $K$ for $y$ near $0$. By the Hopf formula we find, for $t>0$,
$$
u(0,t)= \min_{y\in \R} \left[ \sin \frac 1y + \frac{1}{2t} \right] y^2.
$$
For $t < 1/2$ the expression inside brackets is positive for every $y$, hence the unique minimizer is $y=0$. It follows that $u(0,t) \equiv 0$ for $t \in [0,1/2]$ and that $\gamma(t) \equiv 0$ is a classical characteristic, in accordance with Theorem \ref{one-to-one}. 

On the other hand, if we try to apply the classical method of characteristics, and consider the map
$$
X(y,t):=y+D_pH(u_0'(y))t=  y+u_0'(y)t
$$
we see that, since $u_0''(y)$ is unbounded both from above and below near $y=0$, the map $y \to X(y,t)$ is not monotone near $y=0$, no matter how small $t$ is. Therefore, the problem does not have a $C^1$ solution on $(0,\tau] \times [-\rho,\rho]$ for arbitrarily small $\tau,\rho>0$, as predicted by Theorem \ref{t:semic}. We also notice that the classical characteristic $\gamma$ defined above is entirely contained in the closure of the singular set of $u$, in accordance with Theorem \ref{weak}.

As a side remark, we observe that the solution of this example only exists for $t \in [0,1/2]$, because of the quadratic growth of the initial datum. We could recover the same local behaviour around $(0,0)$ by choosing a globally Lipschitz initial value which coincides with $u_0$  in a neighbourhood of $0$, and in this case the solution would be defined for all times.
}\end{example}

\begin{example} \label{e4}{\rm 
Consider equation \eqref{eq:1+1}
with initial value
$$ u_0(y)=
 \begin{cases}
        -y,\quad &\text{ for }y\leq 0,
        \\
        -|y|^{3/2}\quad &\text{ for } y \geq 0  .
        \end{cases}
$$
Such a function has nonempty proximal subdifferential everywhere: for $y \neq 0$ the proximal subdifferential coincides with the classical derivative, while at the origin we have $D_{pr}^- u_0(y_0) = [-1,0[$. However, since $u_0''(y) \to -\infty$ as $y \to 0^+$, the proximal subdifferentiability does not hold with a uniform $K$.

For any $p_0 \in [-1,0[$, we know from Theorem \ref{one-to-one} that there is a classical characteristic starting from zero of the form
 $\xi_{p_0}(t)=p_0 t$. However, by approximation, we see that there is also a generalized characteristic starting with speed $\dot \xi(0)=0$. Such a characteristic cannot be classical, because $p_0=0$ is not a proximal subgradient. In this case, the point $y=0$ generates a fan of classical characteristics plus a singular generalized one. It can be checked that each characteristic of the form $\xi_{p_0}$ remains classical until it intersects the singular one at a time $t^*(p_0)$, where $t^*(p_0)$ is a decreasing function of $p_0 \in [-1,0[$, satisfying $t^*(p_0) \to 0$
as $p_0 \to 0$.
}\end{example}

\section{Generation of singularities from the initial data}

We now consider the generation from the initial data of singular generalized characteristics, i.e. characteristics which are contained in the singular set $\Sigma$ and not just in its closure. Our result applies to a more specific class of Hamilton-Jacobi equation, with a quadratic Hamiltonian with respect to $p$. More precisely, we make the following structural assumption throughout the section. \medskip

 {\bf (H$^*$)} The Hamiltonian has the form 
 \begin{equation}\label{eq:h}
H(t,x,p)=\frac 12 \langle A(t,x)p,p\rangle +V(t,x),\qquad (t,x)\in [0,T]\times \Omega ,
  \end{equation}
 where  $V:[0,T]\times \Omega \longrightarrow \R$ is a function of class $C^2$ and
 $A(t,x)$ is a family of positive definite matrices with $C^2$ coefficients such that, as quadratic forms,
    $$
c_1 I< A(t,x)<c_2 I , 
    $$
for every $(t,x)\in [0,T]\times \Omega $, for suitable positive constants $c_1,c_2$.
\medskip

For a Hamiltonian as in \eqref{eq:h},  the equation satisfied by a generalized characteristic \eqref{ggc}
takes the simplified form
$$
\dot \gamma(t) \in A(t,\gamma(t)) D_x^+u(t,\gamma(t)), \qquad \mbox{for a.e. }t \in [t_0,t_1].
$$
In this setting, the following result about global propagation of singularities is available.

\begin{theorem}\label{known}
For any $(t_0,x_0) \in  ]0,T]\times \Omega$ there exists a unique generalized characteristic $\gamma:[t_0,\sigma[ \to \Omega$, where either  ${\rm dist}(\gamma(t),\partial \Omega) \to 0$ as $t \to \sigma$ or $\sigma =T$.

In addition, if  $(t_1,\gamma(t_1))$ is singular for $u$ for some $t_1 \geq t_0$, then $(t,\gamma(t))$ is also singular, for all $t \in [t_1,\sigma[$.
\end{theorem}
\begin{proof} The existence of the generalized characteristic is the content of Theorem \ref{gen-char}. Uniqueness follows from the monotonicity of $D^+_xu$, which is a consequence of semiconcavity, by a standard application of Gronwall's lemma, as in Lemma 1 of \cite{ACNS}.

The result of \cite{A2} shows that any singular point is the starting point of a generalized characteristic consisting of singular points. By the forward uniqueness of generalized characteristics, this implies the last part of the statement.
\end{proof}

We can now prove our propagation result

\begin{theorem}\label{generaz} 
Let $u: [0,T[ \times \Omega \to \R$ be a viscosity solution of \eqref{eq:cp}, under assumptions {\rm (H$^*$)}  and {\rm (U$_0$)}. Let $y_0 \in \Omega$ be such that the proximal subgradient $D_{pr}^- u_0 (y_0)$ is empty. Then there exists a generalized characteristic $\gamma:[0,\sigma[ \to \Omega$ starting from $(0,y_0)$ such that $u$ is not differentiable at $(t,\gamma(t))$, for all $t \in ]0,\sigma [$.
\end{theorem}
\begin{proof}
The existence of a generalized characteristic follows from Remark \ref{initial}. We claim that all points $(t,\gamma(t))$ are singular. Suppose in fact that $u$ is differentiable at $(\bar t,\gamma(\bar t))$ for some $\bar t>0$. By Theorem \ref{known}, $u$ is also differentiable at $(t,\gamma(t))$ for all $t\in ]0,\bar t]$. Hence, the restriction of $\gamma$ to the interval $[0,\bar t]$ is a classical characteristic. But then Theorem \ref{one-to-one} would imply that $D^-_{pr}u_0(y_0)$ is nonempty, in contradiction with our assumptions.
\end{proof}

The solution $u$ of Example 2 exhibits the behaviour predicted by the above theorem. In fact, the proximal subdifferential of $u_0$ at $y=0$ is empty, and the straight line $(t,\gamma(t))$ with $\gamma(t) \equiv 0$ is a singular generalized characteristic starting at $0$.

The following characterization is a direct consequence of Theorems \ref{one-to-one} and \ref{generaz}.

\begin{theorem}\label{c2}
  Let $u: [0,T[ \times \Omega \to \R$ be as in the previous theorem, and let $y_0 \in \Omega$. Then the following assertions are equivalent. \\
  {\bf (i)} The proximal subgradient $D_{pr}^- u_0 (y_0)$ is empty. \\
  {\bf (ii)} All the generalized characteristics $\gamma (\cdot )$, with $\gamma (0)=y_0$, are singular (possibly except at $t=0$).
  \end{theorem} 
To conclude our analysis, we consider the case of a semiconcave initial value. In this case, we obtain a statement on the propagation of singularities from point where the initial datum is nondifferentiable, which extend the analogous propagation results for singularities of semiconcave functions starting from interior points, see \cite{AC,CY,Y}.

\begin{theorem}
  Let $u: [0,T[ \times \Omega \to \R$ be a solution of \eqref{eq:cp}, with $u_0$ locally semiconcave on $\Omega$ and the Hamiltonian as in {\rm (H$^*$)}. Let $y_0\in \Sigma (u_0)$. Then, the generalized characteristic starting at $y_0$ is singular. 
\end{theorem} 
\begin{proof}
If the initial value is semiconcave, then the characteristic starting for $(0,y_0)$ is unique, by the same argument recalled in the proof of Theorem \ref{generaz}.
We further observe that, at any $y_0$ where $u_0$ is not differentiable, we have that $D_{pr}^- u_0 (y_0)=\emptyset$.
In fact, by the properties of semiconcave functions, if $D_{pr}^- u_0 (y_0)\not=\emptyset$ then $D^-u_0(y_0)\cap D^+u_0(y_0)=\{ Du_0(y_0)\}$ in contradiction with  $y_0\in\Sigma (u_0)$. We then conclude  by Theorem~\ref{c2}.
  \end{proof} 

{\bf Acknowledgments} The three authors have been partially supported by the Italian group GNAMPA of INdAM (Istituto Nazionale di Alta Matematica).  P.C. acknowledges the MIUR Excellence Department Project awarded to the Department of Mathematics, University of Rome Tor Vergata, CUP E83C18000100006. P.C. and C.S. have been partially supported by the Grant ``Mission: Sustainability'' 2016 (DOmultiage) of University of Rome Tor Vergata.

\end{document}